\documentclass[a4paper, 12pt]{amsart}
\usepackage{amsmath,amsthm,amssymb}
\usepackage[T1]{fontenc}
\usepackage[margin=3.1cm]{geometry}
\usepackage[numbered]{bookmark}
\usepackage{changepage} 
\usepackage[utf8]{inputenc}
\usepackage{graphics, graphicx, caption, xcolor}
\usepackage{amsmath, amsfonts, amsthm, amssymb}
\usepackage{mathrsfs, mathtools, mathabx}
\usepackage{verbatim} 
\usepackage{hyperref}
\usepackage{cleveref}
\usepackage{enumerate}
\usepackage{xcolor}
\usepackage{cite}

\newtheorem{theorem}{Theorem}[section]

\newtheorem{lemma}[theorem]{Lemma}

\newtheorem{corollary}[theorem]{Corollary}

\newtheorem{remark}{Remark}

\newcommand{\Z}{\mathbb{Z}} 
\newcommand{\N}{\mathbb{N}}

\newcommand{\R}{\mathbb{R}}

\title{On disjointness of 3-IETs from mixing and lightly mixing systems}

\author[Berk]{Przemysław Berk}
\address{Faculty of mathematics and computer science, Nicolaus Copernicus University, ul. Chopina 12/18, 87-100 Toru\'n, Poland}
\email{zimowy@mat.umk.pl}

\begin{document}
\maketitle
\begin{abstract}
    In this note we show that all interval exchange transformations of three intervals (3-IETs) are spectrally disjoint from all mixing and disjoint from lightly mixing systems.
    Until now, it was known that almost every 3-IET is rigid, which implies this property.
    However, not all 3-IETs are rigid, hence our result is a strict improvement of current knowledge. We then use this result, among many others, to prove that although some 3-IETs are mildly mixing, they are always disjoint in the sense of Furstenberg from lightly mixing systems.
\end{abstract}
\section{Introduction}
Interval exchange transformations (IETs) play an important role in the theory of dynamical systems.
One of the key reasons for interest in them is the fact that they appear as the first return maps for smooth area-preserving flows on surfaces.
From the work of Masur and Veech (\cite{Masur1982} and \cite{Veech1982}), it follows that almost every interval exchange transformation of any number of intervals is rigid and every rigid transformation is spectrally disjoint from every mixing transformation (see Section \ref{sec: spectral}).
However, it is not true that \emph{every} IET is rigid. For example, Yancey in \cite{Yancey2016} showed that self-similar symmetric 3-IETs are never rigid.

This leads to a natural question whether all IETs are spectrally  disjoint from mixing transformations.
It was formally asked by Farhangi, Moreira and Zelada (see Question 1 in \cite{FMZ2026}).
Spectral disjointness from mixing systems implies in particular absence of the Lebesgue component in the spectrum.
It is worth mentioning that usually the question of the universal (holding for every IET) spectral properties is being attributed to Veech, since in \cite{Veech1984} he states, that the only spectral property known to hold for every IET, is that they have no mixing factors.
This is a result by Katok (see \cite{Katok1980}) and, to the best of our knowledge, remains the only known universal spectral property of IETs.

The main goal of this note is to prove the following.
\begin{theorem}\label{thm: main}
    All 3-IETs are spectrally disjoint from every mixing system. In particular, there is no 3-IET which has a Lebesgue component in its spectrum.
\end{theorem}
This answers the questions mentioned above for interval exchanges of three intervals.
The proof uses limits of Koopman operators and the fact that a 3-IET can be seen as a certain extension of a rotation, which in turn allows the use of the classical Denjoy-Koksma inequality.
Because of this, the question on universal spectral properties of IETs of higher number intervals remains largely open.

As we mentioned already, Katok in \cite{Katok1980} showed that no IET has a nontrivial mixing factor. The following result, which is partially a consequence of Theorem \ref{thm: main}, strengthens this property in the framework of 3-IETs. Let us recall that a dynamical system $T:(X,\mu)\to (X,\mu)$ is \emph{lightly mixing} if for every pair of measurable sets $A$ and $B$ of $\mu$-positive measure, we have
\[
\liminf_{n\to\infty}\mu (A\cap T^nB)>0.
\]
We have the following theorem.
\begin{theorem}\label{thm:main2}
    All 3-IETs are disjoint in the sense of Furstenberg from all lightly mixing systems. In particular, none of the 3-IETs has a lightly mixing factor.
\end{theorem}
It is worth to mentioning that Robertson in \cite{Robertson} showed that linearly recurrent 3-IETs are in fact mildly mixing. This is not a contradiction with Theorem \ref{thm:main2}, since light mixing is a stronger property than mild mixing.

The structure of the paper is as follows: we first recall some elements of spectral theory and how to see 3-IETs as extensions of rotations. Then we prove Theorem \ref{thm: main}. Furthermore, we recall basic definitions and facts regarding joinings, which are followed by the proof of Theorem \ref{thm:main2}.
\section*{Acknowlegdments}
The author is very grateful to Sohail Farhangi for suggesting the problem and several  discussions.
He also acknowledges the support of the NCN OPUS Grant No. 2025/57/B/ST1/00704.

\section{Elements of spectral theory of dynamical systems}\label{sec: spectral}
We now present elementary notions from the spectral theory of the unitary operators.
Since the goal of this article is to study dynamical systems, we present all the definitions in this context, although the majority of them apply also to general unitary operators. 

Consider a measure preserving dynamical system $(X,\mathcal B,\mu, T)$, where $T$ is a $\mu$-measure preserving automorphism of the measure space $(X,\mathcal B,\mu)$.
We will omit the $\sigma$-algebra $\mathcal B$ in the rest of this article, if there is no need to use it. 
The \emph{Koopman operator} $U_T$ is a unitary operator on $L^2_0(X,\mu)=\{f\in L^2(X,\mu)\mid \int f\,d\mu=0\}$ given by the formula 
\[
U_T(f)=f\circ T.
\]
We will study the spectral properties of $U_T$ by studying spectral measures.
Namely, for any $f\in L^2_0(X,\mu)$ we associate the \emph{spectral measure} $\sigma_{T,f}$ on the circle $(\mathbb T,\mathcal B)$ such that its Fourier transform satisfies
\[
\hat\sigma_{T,f}(n)=\langle f\circ T^{-n},f\rangle,
\]
and a \emph{cyclic space} \[
\mathbb Z_T(f):=\overline{\operatorname{span}\{f\circ T^n\mid n\in\Z\}}.
\]
Let $W$ be the unitary operator on $L^2(\mathbb T,\sigma_f)$ which associates to $g=g(z)$ the function $g(z)\cdot z$ (here $z\in\mathbb T$ is a variable).
Then $\Z_T(f)$ is $U_T$ invariant and $U_T|_{\Z_T(f)}$ is isometrically isomorphic to the action of $W$ on $L^2(\mathbb T,\sigma_{T,f})$, that is, there exists a unitary operator $V_{T,f}:\Z_T(f)\to L^2(\mathbb T,\sigma_{T,f})$ such that
\[
V_{T,f}\,\circ U_T|_{\Z_T(f)}=W\circ V_{T,f}. 
\]

By the classical Spectral Theorem, for every measure preserving dynamical system $(X,\mu,T)$, there exists $F\in L^2_0(X,\mu)$ such that for every $f\in L^2_0(X,\mu)$ it holds that
\[
\sigma_{T,f}<\!\!<\sigma_{T,F}.
\]
We denote $\sigma_{T,F}$ by $\sigma_T$ and refer to it as the \emph{maximal spectral type} of $T$.
Many classical ergodic properties have characterisation in terms of the properties of the maximal spectral type.
For example, we have the following
\begin{itemize}
    \item $T$ is ergodic iff $\sigma_T$ has no atom at 1
    \item $T$ is weakly mixing iff $\sigma_T$ is continuous
    \item $T$ is mixing iff $\sigma_T$ is Rajchman, i.e. $\lim_{n\to\infty}|\hat\sigma_{T,f}(n)|=0$.
\end{itemize}
Moreover, we say that $T$ has a \emph{Lebesgue spectrum} if $\sigma_T$ is (equivalent to) Lebesgue measure on $\mathbb T$.
In particular, if $T$ has Lebesgue spectrum, it is mixing. The classical example of an automorphism with Lebesgue spectrum is the standard Bernoulli shift. 

Alternatively, some ergodic properties can be described by studying weak operator limits of Koopman operators.
What is of our interest is the following:
\begin{itemize}
    \item $T$ is mixing iff $\lim_{|n|\to\infty} U^n_T=0$
    \item $T$ is rigid iff there exists an increasing sequence $\{n_k\}_{k=1}^{\infty}$ such that $\lim_{k\to\infty}U^{n_k}_T=Id$, where $Id$ denotes the identity operator.
\end{itemize}

Finally, let us recall that two measure preserving dynamical systems $(X,\mu,T)$ and $(Y,\nu,S)$ are \emph{spectrally disjoint} if
\[
\sigma_T\perp \sigma_S.
\]
In particular, if $T$ and $S$ are \underline{not} spectrally disjoint, then there exist $f\in L^2_0(X,\mu)$ and $g\in L^2_0(Y,\nu)$ and a unitary operator $V:\Z_T(f)\to \Z_S(g)$ such that
\[
V\circ U_{T}|_{\Z_T(f)}=U_{S}|_{\Z_S(g)}\circ V.
\]
Clearly, a rigid transformation cannot be mixing, unless the space is one point, which from now on we exclude from consideration. It is a classical fact that rigid systems are spectrally disjoint from all mixing systems. The next lemma is a generalization of this fact, and it was proven in \cite{Abd2009}

To fully state it, let us recall that for every Laurent polynomial $P(z)=a_{-n}z^{-n}+\ldots +a_mz^m$ and every unitary operator $U$, we consider the operator $P(U)$ defined as
\[
P(U):=a_{-n}U^{-n}+\ldots+a_m U^m.
\]
% Note that if 
% \[
% P(z)=b(z-\eta_1)\ldots (z-\eta_{m+1})(z^{-1}-\kappa_1)\ldots (z^{-1}-\kappa_k),\text{ with } b,\eta_{1},\ldots,\eta_{m+1},\kappa_1,\ldots,\kappa_n\in\mathbb C, 
% \]
% then due to the fact that all powers of $U$ commute, we get that, alternatively, $P(U)$ is given by
% \[
% P(U)=b(U-\eta_1 Id)\circ \ldots\circ(U-\eta_{m+1} Id)\circ (U^{-1}-\kappa_1Id)\circ (U^{-1}-\kappa_{n}Id).
% \]
\begin{lemma}[Lemma 3 in \cite{Abd2009}]\label{lem: criterium}
Let $(X,\mu,T)$ be such that there exists an increasing sequence $(n_k)_{k=1}^{\infty}$ for which we have
\[
U_T^{-n_k}\to P(U_T)\quad \text{on}\quad L^2_0(X,\mu)
\]
in the weak operator topology, where $P:\mathbb C\to \mathbb C$ is a non-zero Laurent polynomial. 
Then $T$ is spectrally disjoint from all mixing systems. %\textcolor{red}{Technically, the paper you cited only shows this result if you assume that $T$ is weakly mixing. It is, of course, trivial to extend the result to the general situation, but this upgrade deserves at a minimum a footnote, and in the worst case a short proof.}
\end{lemma}
It is noteworthy that in \cite{Abd2009}, the author allowed general analytic functions (on the circle) instead of just polynomials in the limit. However, for our purposes, the above statement is enough. Moreover, Lemma 3 in \cite{Abd2009} assumed that $T$ is weakly mixing. However, this assumption can be omitted, since one can simply decompose $L^2(X,\mu)$ into the space generated by eigenfunctions, which gives a discrete part of the spectrum, and its orthocomplement and work on the latter. 
In general proving that the dynamical system satisfies the convergence in the assumption of Lemma \ref{lem: criterium} is not an easy task. However, in the next section, we present a proposition which allows to obtain this kind of convergence for a special class of systems, called integral automorphisms.

\section{3-IETs as integral automorphisms over rotations.} In this section we describe how every non-reducible IET of 3 intervals is either a rotation or is an extension of a rotation in an appropriate sense.

Let us recall that if $I:=[0,1)$, then $T:I\to I$ is an \emph{interval exchange transformation}, if there exists $d\ge 2$ and a partition of $I$ into intervals $I_1,\ldots,I_d$ such that $T$ is a bijection of $I$ which restricted to each $I_j$, is a translation. Then $T$ is uniquely parametrized by a permutation $\pi:\{1,\ldots,d\}\to\{1,\ldots,d\}$ and a vector $\lambda=[\lambda_j]_{j=1}^d$, where $\pi$ describes how the intervals are exchanged by $T$, while $\lambda_j$ is the length of the interval $I_j$. Since $T$ is a piecewise translation, it preserves Lebesgue measure. We always assume that $\pi$ is \emph{non-reducible}, that is $\pi\{1,\ldots,m\}=\{1,\ldots,m\}$ implies $m=d$. Otherwise the dynamics decomposes into at least two disjoint subsystems which can be analyzed separately.

In this note we consider the case when $d=3$. We will focus on symmetric 3-IETs, that is when
\[
\pi(1)=3,\quad\pi(2)=2\quad\text{and}\quad \pi(3)=1.
\]
Otherwise $T$ is a rotation and it is a classical fact that all rotations have discrete spectrum, that is, the maximal spectral type $\sigma_T$ of $T$ is a countable (or finite) sum of Dirac measures. In particular, $T$ is then disjoint from all mixing systems.

Consider a symmetric 3-IET $T$, with length vector $(\lambda_1,\lambda_2,\lambda_3)$. Assume first that $\lambda_3\le \lambda_1$. Then the induced map $R$ obtained by inducing $T$ to $J:=[0,\lambda_1+\lambda_2)$ is a rotation (2-IET) on $J$ given by the length vector $[\lambda_1-\lambda_3, \lambda_2+\lambda_3]$. Then the original map $T$ can be seen as an automorphism of the space $J\times\{0\}\ \cup\ [\lambda_1-\lambda_3,\lambda_1)\times \{1\}$ and the map $T$ acts as $R$ on $(J\setminus [\lambda_1-\lambda_3,\lambda_1))\times\{0\}$, while for $x\in[\lambda_1-\lambda_3,\lambda_1)\times\{0\}$ it takes $(x,0)$ to $(x,1)$ and it sends $(x,1)$ to $(Rx,0)$. In other words, $T$ is an \emph{integral automorphism} over $R$ and the roof function $r(\cdot):=1+\chi_{[\lambda_1-\lambda_3,\lambda_1)}(\cdot)$. Note that we do not exclude the case when $\lambda_1=\lambda_3$. Then $R$ is simply the identity - rotation by 0.

Analogously, if $\lambda_1<\lambda_3$, then by inducing $T$ to interval $[0,\lambda_2+\lambda_3)$, we get that $T$ is an integral automorphism over rotation $R$ given by the length vector $[\lambda_1+\lambda_2,\lambda_3-\lambda_1]$ and the roof function $r(\cdot):=1+\chi_{[0,\lambda_1)}(\cdot)$. The inducing that we used is nothing else but the classical Rauzy-Veech induction, an extremely useful renormalization tool used to study the properties of IETs for \emph{typical} parameters. Since we want to study the properties of all 3-IETs, we will not use this object in more depth.

Representing a 3-IET by its integral automorphism model allows to study their properties from the underlying rotations. Thus, let us recall the basic notions and facts about the latter.

Let $R_{\alpha}:\mathbb T\to\mathbb T$ be a rotation by $\alpha\in[0,1)$, i.e.
\[
R_{\alpha}(x)=x+\alpha\ \operatorname{mod}\,1,
\]
or alternatively, it is a 2-IET given by a length vector $[1-\alpha,\alpha]$. If $\alpha\in \mathbb Q$, that is $\alpha=\frac{p}{q}$ with $(p,q)=1$, then every point is periodic with period $q$. If $\alpha$ is irrational, then $R_{\alpha}$ is uniquely ergodic w.r.t. Lebesgue measure. %and the rotation admits a partition of the interval $[0,1)$ into Rokhlin towers of intervals, a partition which we now recall.

Let $\alpha\in [0,1)\setminus \mathbb Q$ and let $\alpha=[0;a_1,a_2,\ldots)$ be its continued fraction expansion. Let $(q_n)_{n=1}^{\infty}$ be the sequence of denominators of convergents of $\alpha$. Then for every $x\in\mathbb T$ we have
\begin{equation}\label{eq: rotationrigidity}
    d\big((R_{\alpha})^{q_n}x,x\big)=\|q_{n}\alpha\|\to 0\quad \text{as }\ n\to\infty,
\end{equation}
where $d$ is the Euclidean distance on the circle and $\|\cdot\|$ is the distance to the nearest integer.

We finish this section by stating the result by El Abdalaoui, Lema\'nczyk and de la Rue, which yields the convergence of Koopman operators for integral automorphisms, as needed to apply Lemma \ref{lem: criterium}.
\begin{theorem}[Theorem 1 in \cite{ALR2014}]\label{thm: operator_limits}
Let $S$ be an ergodic $\mu$-measure preserving Borel automorphism of a compact metric space $(X,d)$. Let also $f:X\to\Z_+$. Assume that there exists a sequence $(q_n)_{n=1}^{\infty}$ such that 
\begin{equation}\label{eq: rigidity}
\sup_{x\in X}\,d(S^{q_n}x,x)\to 0\quad\text{as}\ n\to\infty.
\end{equation}
Assume also that there exists an increasing integer sequence $(h_n)_{n=1}^{\infty}$ such that
\begin{equation}\label{eq: BSconvergence}
\lim_{n\to\infty}(\sum_{k=0}^{q_n-1}f\circ S^k-h_n)_*\mu= P\quad \text{weakly,}     
\end{equation}
for some probability measure $P$ on $\Z$.
Then, if $S^f$ is the integral automorphism given by $S$ and $f$, the following convergence holds:
\[
U_{S^f}^{-h_n}\to \sum_{k\in\Z}P(\{k\})U_{S^f}^k,
\]
in the weak operator topology.
\end{theorem}
As evidenced by the above theorem, to study the limits of Koopman operators, we need to control the Birkhoff sums of the roof function $r$, with respect to the underlying rotation. This will be the core of the proof of Theorem \ref{thm: main}, to which the next section is devoted.

\section{Limit distributions of Birkhoff sums - proof of Theorem \ref{thm: main}}

Let $T:[0,1)\to[0,1)$ be a symmetric IET of three intervals $I_1, I_2$ and $I_3$, given by the length vector $[\lambda_1,\lambda_2,\lambda_3]$. As shown in the previous section, $T$ can be represented as an integral automorphism over a rotation $R_{\alpha}$ for some $\alpha\in[0,1)$ and a roof function of the form $r:=1+\chi_{E}$, where $E$ is a left-closed right-open subinterval of $[0,1)$ (up to proper rescaling).  
% interval of the form $[1-\alpha,y)$ with $y>1-\alpha$ or $[0,y)$ with $y<1-\alpha$, when $\lambda_1\ge \lambda_3$ or $\lambda_1<\lambda_3$, respectively. For the rest of the proof we assume that $\lambda_1\ge\lambda_3$ holds, that is we consider the first case. The second case goes along the same lines.

Let then $R_{\alpha}^r$ be the integral automorphism, isomorphic to the initial IET, with $\alpha\in[0,1)$ and $r=1+\chi_{E}$, where $E\subset \mathbb T$ is a subinterval of length $|E|$. We consider cases depending first on the rationality of $\alpha$.
%, and then, in the irrational case, the relative position of the point $y$ w.r.t. the orbit of %$0$.
\subsection{$\alpha$ is rational.} Assume first that $\alpha\in[0,1)\cap \mathbb Q$, i.e. $\alpha=\frac{p}{q}$, with $(p,q)=1$. First, note that for every $x\in\mathbb T$, the number of points from the orbit $\{x,R_{\alpha}x,\ldots,R^{q-1}_{\alpha}x\}$ which belong to $E$ is either $\lfloor q|E|\rfloor$ or $\lceil q|E|\rceil$. Then, for $R_{\alpha}^r\sim T$, every point is $(q+\lfloor q|E|\rfloor)\cdot (q+\lceil q|E|\rceil)$-periodic. Thus $T$ has a discrete spectrum and is hence spectrally disjoint from  weakly mixing systems, and in particular, from every mixing system.

\subsection{$\alpha$ is irrational.} Let $\{q_n\}$ be the sequence of denominators of convergents of $\alpha$. Note that $r$ is a function of bounded variation, with $Var(r)=2$. Thus, by the classical Denjoy-Koksma inequality, for every $n\in\N$ and every $x\in\mathbb T$ we have
\[
-2\le \sum_{j=0}^{q_n-1}r((R_\alpha)^jx)-q_n(1+|E|)\le 2
\]
and thus
\[
-3\le \sum_{j=0}^{q_n-1}r((R_\alpha)^jx)-\lfloor q_n(1+|E|)\rfloor\le 3.
\]
By Prokhorov's Theorem, passing to a subsequence if necessary yields
\[
\Big(\sum_{j=0}^{q_n-1}r((R_\alpha)^jx)-\lfloor q_n(1+|E|)\rfloor \Big)_*Leb\to P,
\]
where $P$ is a probability measure on $\mathbb Z$ with $P\{-3,-2,-1,0,1,2,3\}=1$. This, together with \eqref{eq: rotationrigidity}, means that we can apply Theorem \ref{thm: operator_limits} with $h_n:=\lfloor q_n(1+|E|)\rfloor$ and get that 
\[
U_{R^r_{\alpha}}^{-h_n}\to \sum_{k=-3}^3P(\{k\})U_{R^r_{\alpha}}^{k},
\]
which in view of Lemma \ref{lem: criterium} gives spectral disjointness from all mixing systems and finishes the proof of Theorem \ref{thm: main}.

\section{Joinings and symmetric products}
We will now pass to the notions necessary to prove (and properly state) Theorem \ref{thm:main2}. Let us recall that given two probability preserving systems $(X,\mathcal B,\mu,T)$ and $(Y,\mathcal C,\nu,S)$, a \emph{joining} $\eta$ is a probability measure on $(X\times Y,\mathcal B\otimes\mathcal C)$ such that
\begin{enumerate}
    \item the marginals of $\eta$ on $X$ and $Y$ are $\mu$ and $\nu$ respectively, and
    \item $\eta$ is $T\times S$-invariant.
\end{enumerate}
Note that $\mu\otimes \nu$ is always such a measure. We say that $T$ and $S$ are \emph{disjoint in the sense of Furstenberg} if $\mu\otimes \nu$ is the only joining between $T$ and $S$. 

For the purpose of the proof of Theorem \ref{thm:main2}, we will put an additional focus on the \emph{self-joinings}, i.e. joinings of two identical systems $(X,\mathcal B,\mu,T)$. We will denote such joinings as $\mathcal J_2(T)$. Moreover we extend the definition of a joining into an arbitrary amount of coordinates $k\ge 2$. We then refer to them as \emph{$k$-self-joinings} and denote by $\mathcal J_k(T)$. We also denote by $\mu^{\otimes k}$, a product measure of $k$ copies of $\mu$. 

For any $k\ge 2$ (or $k=\infty$), we say that $\eta\in \mathcal J_k(T)$ is an \emph{off-diagonal joining} if 
\[
\eta=(\varphi_1\times\ldots\times\varphi_k)_*\mu,\footnote{Here $\varphi_1\times\ldots\times\varphi_k:X\to \R^{k}$ and it is given by $\varphi_1\times\ldots\times\varphi_k(x)=(\varphi_1(x),\ldots,\varphi_k(x))$.}
\]
where $\varphi_1,\ldots,\varphi_k$ are elements of $C(T)$, the \emph{centralizer} of $T$. Naturally, if $T^{\times k}$ denotes the product of $k$ copies of $T$, and $\eta\in \mathcal J_k(T)$ is an off-diagonal joining, then $(T^{\times k},\eta)$ is isomorphic to $(T,\mu)$, with isomorphism given by $\varphi_1\times\ldots\times\varphi_k$. We say that $T$ is \emph{$k$-simple} if every ergodic joining in $\mathcal J_k(T)$ is either an off-diagonal joining or an independent product of off-diagonal joinings. 

Let us also recall a stronger notion than 2-simplicity. Namely, we say that $(X,\mu,T)$ has \emph{minimal self-joinings} if every ergodic joining $\eta\in \mathcal J_2(T)$ is either
\begin{itemize}
    \item a product joining $\mu\otimes\mu$, or
    \item it is an off-diagonal joining of the form $(Id\times T^n)_*\mu$ for some $n\in\Z$.
\end{itemize}
Note that if $T$ has minimal self-joinings, then it cannot have nontrivial factors. Indeed, otherwise a relative independent extension over this factor yields a nontrivial ergodic self-joining. Consequently, $T$ is ergodic, since it would otherwise have a nontrivial identity factor.

We are now able to state the main result concerning simplicity, that is going to be of major use to us.
\begin{theorem}[Theorem 4 in GHR]\label{thm:GHR}
Let $(X,\mu,T)$ be a $2$-simple weakly mixing system which does not have a factor whose spectrum is Lebesgue. Then $(X,\mu,T)$ is $k$-simple for every $k\ge 2$.
\end{theorem}

The reason why we need the above theorem comes from the next fact that concerns symmetric powers. Let us recall that given $(X,\mu,T)$ for $k\in\N$, a \emph{k-symmetric power} is the map $T^{\odot k}$ on $X^{\odot k}$, where $X^{\odot k}=\{\{x_1,\ldots,x_k\}\mid x_1,\ldots,x_k\in X\}$, given by 
\[
T^{\odot k}\{x_1,\ldots,x_k\}=\{T(x_1),\ldots, T(x_k)\}.
\]
Note that $T^{\odot k}$ is a factor of $T^{\times k}$, with the factorizing map $\pi$ given simply by formula $\pi(x_1,\ldots,x_k)=\{x_1,\ldots,x_k\}$, with the invariant measure taken naturally as $\mu^{\odot k}:=\pi_*\mu^{\otimes k}$ i.e. the projection of the product measure via the factorizing map. The following result is the most important general fact that will serve as the main tool to prove Theorem \ref{thm:main2}.

\begin{theorem}[Theorem 4.1 in \cite{dJR}]\label{thm:dJR}
    Assume that $(T,\mu)$ is simple of all orders with a trivial centralizer and $(S,\nu)$ is ergodic. If $T$ and $S$ are not disjoint, then there exists $k\in\N$ such that $(T^{\odot k},\mu^{\odot k})$ is a factor of $(S,\nu)$.
\end{theorem}
We remark that in \cite{dJR}, the authors do not need the assumption on trivial centralizer but then the factor in the conclusion is actually a symmetric power of $T/K$ on $X/K$, where $K\subset C(T)$ is a compact subgroup of the centralizer of $T$. However, under the assumption of trivial centralizer, i.e. $C(T)=\{T^k\mid k\in\Z\}$, the only compact subgroup of $C(T)$ is $\{Id\}$.

Theorem \ref{thm:dJR} gives a very strong conclusion under a strong assumption of simplicity of all orders. In the next section, we will recall a fact by Ferenczi, Holton and Zamboni, which, together with Theorem \ref{thm: main} gives that 3-IETs are either rigid or simple of all orders, which perfectly fits into our setup.

We will finish this section by showing the following simple lemma, which will be used in the final step of the proof of Theorem \ref{thm:main2}.
\begin{lemma}\label{lem:notlmsymmetric}
    If $(X,\mathcal B,\mu,T)$ is not lighlty mixing, then for any $k\in\N$, the symmetric power $T^{\odot k}$ is also not lightly mixing.
\end{lemma}
\begin{proof}
    Since $T$ is not lightly mixing, there exists a set $A\in\mathcal B$, with $\mu(A)>0$
    such that 
    \[
    \liminf_{n\to\infty}\mu(A\cap T^n B)=0.
    \]
    Then if for $C\in\{A,B\}$, we denote $C^{\odot k}:=\{\{x_1,\ldots,x_k\}\mid x_1,\ldots,x_k\in C\}$ and $C^{\times k}:=\{(x_1,\ldots,x_k)\mid x_1,\ldots,x_k\in C\}$, then
    \[
    \begin{split}
    \liminf_{n\to\infty} &\mu^{\odot k}\big(A^{\odot k}\cap (T^{\odot k})^n(B^{\odot k})\big)=\liminf_{n\to\infty} \mu^{\otimes k}\big(A^{\times k}\cap (T^{\times k})^n(B^{\times k})\big)\\
    &\le \liminf_{n\to\infty}(A\cap T^nB)=0,
    \end{split}
    \]
    which gives that $T^{\odot k}$ is not lightly mixing.
\end{proof}

\begin{comment}{\begin{remark}\label{rem:infjoinings}
    It is easy to prove by induction, that if $(X,\mu,T)$ is simple of all orders, for every ergodic joining $\eta\in \mathcal J_{\infty}(T)$, there exists a (possibly finite) partition $P_1\sqcup P_2\sqcup P_3\sqcup\ldots=\mathbb N$, such that for every $j=1,2,\ldots$ the projection of $\eta$ to coordinates given by $P_j$ is an off-diagonal joining and $\eta$ is a product of such joinings.
\end{remark}}\end{comment}

\section{3-IETs are disjoint from lightly mixing systems - proof of Theorem \ref{thm:main2}}
The proof of Theorem \ref{thm:main2} naturally breaks into two cases, given by the following theorem of Ferenczi, Holton and Zamboni.
\begin{theorem}[Corollary 5.2 in \cite{FHZ}]\label{thm:FHZ}
    Every 3-IET $T$ is either 
    \begin{enumerate}
    \item rigid, or
    \item has minimal self-joinings.
    \end{enumerate}
\end{theorem}
\begin{remark}
    In fact, in \cite{FHZ}, the authors proved the above dichotomy under the assumption that the IET $T$ satisfies the \emph{Keane condition} (see \cite{Keane75} for definition). However, Bernazzani in \cite{Bernazzani} proved, that if this condition is not satisfied, then the 3-IET under consideration is a finite constant suspension (,,a tower'') of a rotation, which is again, a rigid transformation.
\end{remark}

\subsection{Rigid 3-IET} Assume first that $T$ is a rigid 3-IET. Let us recall that a probability preserving system $(X,\mu,S)$ is \emph{mildly mixing} if for every measurable set $A\subset X$ with $\mu(A)>0$ we have
\[
\liminf_{n\to\infty}\mu(A\triangle T^nA)>0.
\]
or, equivalently
\[
S\text{ is disjoint from every rigid transformation}.
\]
In particular, every lightly mixing system is mildly mixing. Then, if $T$ is rigid, it is disjoint from every mildly mixing system and thus also from every lightly mixing system. This finishes the proof of Theorem \ref{thm:main2} in the case when $T$ is rigid.

\subsection{3-IET with minimal self-joinings.} We now consider the harder case, when the 3-IET $T$ has minimal self-joinings. For this purpose, let us recall the following theorem by Boshernitzan and Chaika.
\begin{theorem}[Theorem 12.1 in \cite{BoshernitzanChaika}]\label{thm:BC}
    No 3-IET is topologically mixing.
\end{theorem}
We have the following simple consequence.
\begin{corollary}\label{cor:nolm}
    No 3-IET is lightly mixing.
\end{corollary}
\begin{proof}
    Let $T$ be a 3-IET. By Theorem \ref{thm:BC}, there exist two open sets $A,B\in [0,1)$ and an increasing sequence $\{n_k\}_{k\in\N}$ such that
    \[
    A\cap T^{n_k}B=\emptyset.
    \]
    Thus
    \[
    \liminf_{n\to\infty}Leb(A\cap T^{n}B)=0,
    \]
which shows that $T$ is not lightly mixing.
\end{proof}
\begin{remark}
    Note that, since we assume that $T$ has minimal self-joinings, it has no non-trival factors. In particular, in view of Corollary \cite{cor:nolm}, it has no lightly mixing factors. However, this is a weaker property than being disjoint from a lightly mixing system.
\end{remark}
\begin{comment}{To prove that $T$ is disjoint from all lightly mixing systems, we will use the following criterion by Lema\'nczyk, Parreau and Thouvenot. 

\begin{theorem}[Theorem 4 in \cite{LPT}]\label{thm:LPT}
If $(X,\mathcal B,\mu,T)$ and $(Y,\mathcal C,\nu,S)$ are two ergodic probability preserving systems, then there exists an ergodic joining $\eta\in\mathcal J_{\infty}(T)$ such that $(Y,\nu,S)$ and $(X^{\times\infty},\eta,T^{\times\infty})$ have a nontrivial common factor.
\end{theorem}}\end{comment}

Let now $(Y,\nu,S)$ be a lightly mixing system. Assume by contradiction that $S$ and $T$ are not disjoint. Then, if we can apply Theorem \ref{thm:dJR}, there exists $k\in\N$ such that $T^{\odot k}$ is a factor of $S$. Since a factor of a lightly mixing is also lightly mixing but in view of Corollary \ref{cor:nolm} and Lemma \ref{lem:notlmsymmetric}, the map $T^{\odot k}$ is not lightly mixing, we get a contradiction.

Thus, it remains to show that we can use Theorem \ref{thm:dJR}, i.e. $T$ satisfies its assumptions. First, since $T$ has minimal self-joinings, it has no non-trival factors (in particular it is weakly mixing, since otherwise it would have rotation as a factor), it has a trivial centralizer, and it is 2-simple. Moreover, in view of Theorem \ref{thm: main}, $T$ has no Lebesgue component in the spectrum. Thus, by Theorem \ref{thm:GHR}, $T$ is simple of all orders. This shows that we can apply Theorem \ref{thm:dJR} and finishes the proof of Theorem \ref{thm:main2}.

\section*{AI acknowledgement} A generative LLM was used for proofreading and literature
searching. The initial version of this article and all main mathematical contributions are from the author.

\end{document}